\documentclass[a4paper]{article}

\usepackage[english]{babel}
\usepackage[latin1]{inputenc}
\usepackage{epsfig}
\usepackage{amsthm,amssymb,amsbsy,amsmath,amsfonts,amssymb,amscd,mathrsfs}
\usepackage{colortbl}
\usepackage{verbatim}
\usepackage{url}

\usepackage{mathtools}          
\usepackage{amssymb}
\usepackage{stmaryrd}           
\usepackage{color}
\usepackage{graphics}
\usepackage{graphicx}
\usepackage{epsfig}
\usepackage{wrapfig}
\usepackage{tikzgraphicx}
\usetikzlibrary{hobby}
\usepackage[f]{esvect}
\usepackage{multirow,enumerate}

\usepackage{amsthm}%

\newtheorem{remark}{Remark}[section]
\newtheorem{proposition}{Proposition}[section]

\newtheorem{defi}{Definition}[section]

\newcommand{\norme}[1]{\lVert#1\rVert}
\newcommand{\diff}{\,{\rm d}}
\def\xCn#1{{\rm C}^#1}

\newcommand{\nbSet}[1]{\mathbb{#1}}
\newcommand{\R}{\nbSet{R}} 

\newcommand{\uset}{\mathcal{U}}     

\newcommand{\intervalle}[4]{\mathopen{#1}#2
                                \mathclose{}\mathpunct{},#3
                                \mathclose{#4}}
\newcommand{\intervalleff}[2]{\intervalle{[}{#1}{#2}{]}}

\newcommand{\intervallefo}[2]{\intervalle{[}{#1}{#2}{)}}

\newcommand{\enstq}[2]{\left\{#1\mathrel{}\middle|\mathrel{}#2\right\}}

\newcommand{\Htrue}{\mathbf{H}}

\newcommand{\liebra}[2]{{[}{#1},{#2}{]}}

\definecolor{blue}{rgb}{0.1 0.1 0.8}
\definecolor{red}{rgb}{0.95 0.1 0.1}
\definecolor{green}{rgb}{0.2 0.6 0.2}
\definecolor{blueDrift}{rgb}{0.1 0.1 0.8}
\definecolor{redDrift}{rgb}{0.8 0.1 0.1}
\definecolor{greenDrift}{rgb}{0.3 0.8 0.3}

 \usepackage{circuitikz}
 \usetikzlibrary{arrows,shapes,backgrounds,patterns}
 \usepackage{pifont}
 \usepackage{tikzgraphicx}
 \tikzstyle{every picture}+=[remember picture]
 \tikzstyle{na} = [baseline=-.5ex]
 \makeatletter
 \newcommand{\gettikzxy}[3]{%
  \tikz@scan@one@point\pgfutil@firstofone#1\relax
  \edef#2{\the\pgf@x}%
  \edef#3{\the\pgf@y}}%
\usepackage{textcomp}   
\DeclareMathOperator{\sign}{sign}

\DeclareMathOperator{\atan}{atan}

\newcommand{\MBall}{\mathbb{B}}                       
\newcommand{\MSphere}{\mathbb{S}}

\title{Abnormal Geodesics in 2D-Zermelo Navigation Problems in the Case of Revolution and the Fan Shape of the Small Time Balls}

\author{B.~Bonnard\footnote{Inria, 2004 route des Lucioles F-06902, Sophia Antipolis, France, bbonnard@u-bourgogne.fr.},
O.~Cots\footnote{Toulouse Univ., INP-ENSEEIHT-IRIT, UMR CNRS 5505, 2 rue Camichel, 31071 Toulouse, France, olivier.cots@irit.fr.},
J.~Gergaud\footnote{Toulouse Univ., INP-ENSEEIHT-IRIT, UMR CNRS 5505, 2 rue Camichel, 31071 Toulouse, France, joseph.gergaud@irit.fr.}
 and
B.~Wembe\footnote{Toulouse Univ., IRIT-UPS, UMR CNRS 5505, 118 route de Narbonne, 31062 Toulouse, France, boris.wembe@irit.fr.}%
}

\begin{document}

\maketitle

\begin{abstract}
In this article, based on two case studies, we discuss the role of abnormal geodesics in planar Zermelo navigation problems. Such curves are limit curves of the accessibility set, in the domain where the current is strong. The problem is set in the frame of geometric time optimal control, where the control is the heading angle of the ship and in this context, abnormal curves are shown to separate time minimal curves from time maximal curves and are both small-time minimizing and maximizing. We describe the small-time minimal balls. For bigger time, a cusp singularity can occur in the abnormal direction, which corresponds to a conjugate point along the non-smooth image. It is interpreted in terms of the regularity property of the time minimal value function.
\end{abstract}




\section{Introduction}

A Zermelo navigation problem in the plane can be stated using \cite{BrysonHo:1975}
as a time minimal control problem described by the dynamics
\begin{equation}
    \dot{q}(t) = F_0(q(t)) + \sum_{i=1}^{2} u_i(t) F_i(q(t))
    \label{eq:scl_dynamics}
\end{equation}
where $q = (x, y)$ are the coordinates, $F_i$ being $\xCn{\infty}$-vector fields and
$u = (u_1, u_2)$ is the control, bounded by $\norme{u} = \sqrt{u_1^2 + u_2^2}
\le 1$. The vector field $F_0$ is called the \emph{current} (or drift) while
the \emph{control fields} $F_1$ and $F_2$ define a \emph{Riemannian metric} $g$,
taking $\{F_1, F_2\}$ as an orthonormal frame. If $\norme{u} = 1$, one can set 
$u = (\cos \alpha, \sin \alpha)$ and $\alpha$ is the \emph{heading angle} of the ship.

This type of problems originated from a \emph{historical example} in the frame of calculus of variations, due to Carath\'eodory and Zermelo and where a rather complete analysis was presented
in Refs.~\cite{Caratheodory, Zermelo:1931}. This historical example is one of the 
motivations of this article.
In this example, the dynamics is described by
\begin{equation*}
    F_0(q) = y \frac{\partial }{\partial x}, \quad g = \diff x^2 + \diff y^2,
\end{equation*}
where $g$ is the \emph{Euclidean metric}. Taking an arbitrary metric and considering 
the \emph{weak current case} $\norme{F_0}_g < 1$, with $\norme{\cdot}_g$ the norm associated
to the metric $g$, this problem leads to a Zermelo navigation problem in Finsler geometry
\cite{BaoRoblesShen:2004}.

More recently, in \cite{BCW2019}, a Zermelo navigation problem was analyzed in details,
associated to the evolution of a \emph{passive tracer}, where the current is related to a
\emph{vortex}, centered at the origin of the reference frame.
This problem comes at the origin from hydrodynamics \cite{ArnoldKhesin:1998}. 
Moreover, it is a toy model for the N-body problem, in the frame
of Hamiltonian dynamics \cite{MeyerHallOffin:2009}. In this case, the system evolves on the
punctured plane $\R^2 \setminus \{0\}$, the current being given by
\begin{equation*}
    F_0(q) = \frac{k}{(x^2+y^2)} \left( -y \frac{\partial}{\partial x} + x \frac{\partial}{\partial y} \right)
\end{equation*}
where $k > 0$ is the \emph{circulation parameter} and the control fields being
given by $g = \diff x^2 + \diff y^2$. Using polar coordinates
$q = (r, \theta)$ one has
\begin{equation*}
    F_0(q) = \frac{k}{r^2}\frac{\partial}{\partial \theta}
\end{equation*}
and the Euclidean metric takes the form:
$
    g = \diff r^2 + r^2 \diff \theta^2.
$

These two cases can be set in the same geometric frame by considering in a coordinate system $q = (r, \theta)$, a Zermelo
navigation problem, where the current is in the form 
\begin{equation}
\label{eq:scl_general_current}
    F_0(q) = \mu(r) \frac{\partial}{\partial \theta}
\end{equation}
and where the metric is given by
\begin{equation*}
    g = \diff r^2 + m^2(r) \diff \theta^2, \quad m(r) > 0.
\end{equation*}
Such a metric was the object of many studies in the context of mechanics and Riemannian geometry
and it is called a \emph{metric of revolution} in \emph{Darboux coordinates} $(r, \theta)$, 
where the lines $r = constant$ are the \emph{parallels} and the lines $\theta = constant$
are the \emph{meridians} \cite{Bolsinov:2004}. Note that in \eqref{eq:scl_general_current}, 
the current is along the parallels only, which is sufficient to cover the two founding examples.
We refer to \cite{HKS:2018} for a case study in the differential geometric frame, 
in the case of a \emph{weak current}, that is \emph{Randers problems} in Finsler geometry \cite{BaoRoblesShen:2004}, assuming $\norme{F_0}_g < 1$.
In this article, we focus on the case of a \emph{strong current}, that is $\norme{F_0}_g > 1$.
It was already the case studied in details in the historical example. 
On the opposite, in the vortex problem \cite{BCW2019}, the analysis was concentrated on the
situation where at the initial point, the current is weak, but the tracer can reach the vicinity
of the vortex where the current is strong.

To present our contribution, we must introduce the following concepts from geometric optimal
control. The set of admissible controls $\uset$ is the set of measurable mappings $u$ from $\intervallefo{0}{+\infty}$ to the unit closed Euclidean ball:
$\norme{u} \le 1$, endowed with the $L^\infty$-norm topology.
We denote by $q(\cdot, q_0, u)$ the response of the dynamics \eqref{eq:scl_dynamics} associated to
$u$, with $q(0) = q_0$. Let $t_f \ge 0$, the \emph{accessibility set} from $q_0$ in time $t_f$
is defined by $A(q_0, t_f) = \enstq{q(t_f, q_0, u)}{u \in \uset}$ (if $q(\cdot, q_0, u)$
is defined on the whole $\intervalleff{0}{t_f}$) and this gives the \emph{accessibility set} from
$q_0$ defined by $A(q_0) = \cup_{t_f \ge 0} A(q_0, t_f)$. Thanks to existence theorems in optimal
control, in many cases (and in particular in the two case studies), for each $q_0$, $q_1$
provided, ($q_0$, $q_1\neq 0$ in the vortex case), there exists a time minimal solution to
transfer $q_0$ to $q_1$, and from the \emph{Maximum Principle} \cite{Pontryagin:1962}, candidates
as minimizers are geodesic curves. If such a theorem holds, fixing $q_0$, the \emph{time minimal 
value function} is given by: 
\[
T(q_1) = \inf \enstq{t_f}{q(t_f,q_0,u) = q_1 \;\; \text{and} \;\; u \in \uset}
\]
and the \emph{sphere} $\mathbf{S}(q_0,r)$ of radius $r$ is the set of points $q_1$ which can be reached from $q_0$ in minimum time $r$, while the \emph{ball} of radius $r$ is the set $\mathbf{B}(q_0,r) = \bigcup_{r'\leq r} S(q_0,r')$. 

The aim of this article is double. First of all, and based on \cite{BonnardKupka:1993} we provide the geometric frame from optimal control theory to analyze such Zermelo navigation problems and we make a focus on the role of abnormal geodesics (the limit curves in Carath\'eodory terminology) in such problems. One ingredient is to introduce the \emph{Carath\'eodory-Zermelo-Goh transformation} which amounts to parameterize the geodesics using as accessory control the derivative of the heading angle. This allows to evaluate the accessibility set and its boundary filled by the geodesics, in the abnormal directions, as the image of the extremity mapping, using \emph{semi-normal forms}. 
The second step is to extend this analysis to larger times as shown by the two case studies. A singularity can occur along the abnormal geodesic corresponding to a \emph{cusp point} and associated to a concept of \emph{conjugate point} along non-smooth image of the abnormal geodesic, extending the concept of conjugate point in the smooth case, introduced in \cite{BonnardKupka:1993}. This leads to describe the regularity of the time minimal value function (from generic point of view) in both normal and abnormal cases.
See \cite{CanarsaSinestari:2004,CanarsaRifford:2008} for the relation with singularities of semi-concave functions.

The organization of this article is as follows. In section \ref{sec:scl_PMP}, we introduce the geometric frame from optimal control to analyze in a general case the time minimal solutions, in the context of Hamiltonian dynamics, using the Maximum Principle \cite{Pontryagin:1962}. This provides the parameterization of both \emph{time minimal and time maximal} solutions, which is crucial to understand the role of abnormal geodesics. We introduce the Carath\'eodory-Zermelo-Goh transformation which amounts to extend our dynamics to a \emph{single-input affine system}, and therefore to use the results from \cite{BonnardKupka:1993}, in particular to construct the \emph{tangent model} in order to clarify the role of abnormal geodesics. This from a geometric point of view, in relation with \emph{Lie algebraic computations}. This allows to evaluate the extremity mapping in the abnormal directions, image of the exponential mapping, and its boundary using the concept of Jacobi field. In section \ref{sec:scl_cusp}, using the symmetry of revolution, the geodesics curves can be parameterized, thanks to integrability properties. This is the tool to compute the sphere and ball of general radius and to describe the time minimal value function. In particular, this allows to analyze in a more general context two important features observed in the historical example: the \emph{existence of a cusp singularity for abnormal geodesics} (related to a phenomenon of self-intersections of neighboring geodesics) and the \emph{non-continuity of the value function}. This property is shown to appear when the abnormal geodesics meet the boundary between the two domains defined by $\norme{F_0(q)}_g=1$. This causes a morphogenesis of the accessibility set \cite{Thom:1976}. In the conclusion we present a program of further studies related to this note. First, \emph{models of singularities} of the value function are described, and this can be used in a more general context, since integrability is not a crucial issue in our study. Second, Zermelo navigation problems with symmetry of revolution are an important geometric object of study, in the frame of \emph{integrable} Hamiltonian dynamics, in relation to mathematical physics.

\section{Maximum Principle and evaluation of the accessibility set in the regular geodesic case}
\label{sec:scl_PMP}

\subsection{Maximum Principle}

Consider the Zermelo navigation problem whose dynamics is described by~\eqref{eq:scl_dynamics}. To formulate the Maximum Principle \cite{Pontryagin:1962}, we introduce the \emph{pseudo-Hamiltonian} associated to the cost (extended) system: 
\begin{equation*}
    H(z,u) = H_0(z) + u_1 H_1(z) + u_2 H_2(z) + p^0
\end{equation*}
where $z=(q,p)$, $p=(p_x,p_y)$ being the \emph{adjoint vector}, $H_i(z)=p\cdot F_i(q)$ being, for $i=0,1,2$, the Hamiltonian lift of the vector field $F_i$ ($\cdot$ denotes the standard inner product) and $p^0$ is a constant.
The \emph{maximized} (or true) \emph{Hamiltonian} is given by the maximization condition 
\begin{equation*}
    \Htrue(z) = \max_{\norme{u}\leq 1} H(z,u), 
\end{equation*}
and since $F_1$, $F_2$ form a frame, we have

\begin{proposition}
\begin{enumerate}
\item The maximizing controls are given by 
\begin{equation}
\label{eq:scl_optimal_control}
u_i(z) = \frac{H_i}{\sqrt{H_1^2 + H_2^2}} , \quad i = 1,2.
\end{equation}
\item The maximized Hamiltonian is 
    $\Htrue(z) = H_0(z) + \sqrt{H_1^2+H_2^2} + p^0 = 0.$
%
\item Candidates as time minimizers (resp. maximizers) are solutions of: 
\begin{equation}
\label{eq:scl_Hamiltonian_dynamics}
    \dot{z}(t) = \vv{\Htrue}(z(t)),
%
\quad \text{with} \quad
\vv{\Htrue} = \frac{\partial \Htrue}{\partial p}\frac{\partial}{\partial x} - \frac{\partial \Htrue}{\partial x}\frac{\partial}{\partial p},
\end{equation}
and $p^0 \leq 0$ (resp $\geq 0$) in the time minimal (resp. maximal) case.
\end{enumerate}
\end{proposition}

\begin{defi}
An \emph{extremal} is a solution $z(\cdot) = (q(\cdot),p(\cdot))$ of \eqref{eq:scl_Hamiltonian_dynamics} and a projection of an extremal is called a \emph{geodesic}. It is called \emph{strict} if $p$ is unique up to a factor, \emph{normal} if $p^0 \neq 0$ and abnormal (or exceptional) if $p^0=0$. In the normal case it is called \emph{hyperbolic} (resp. \emph{elliptic}) if $p^0<0$ (resp. $p^0>0$).
\end{defi}

Next we relate geodesic curves to singularities of the extremity mapping, which is an important issue in our analysis, see \cite[Chapter~3]{BonnardChyba:2003} for a general context and details.

\begin{defi}
Restrict the control domain to the set $\norme{u}=1$. Let $q(\cdot,q_0,u)$ be the response of $u$, with $q(0) = q_0$. Fixing $q_0$, the \emph{extremity mapping} is the map: $E^{q_0,\cdot} \colon u \mapsto q(\cdot, q_0,u)$ and the \emph{fixed time extremity mapping} (at time $T$) is the map: $E^{q_0,T} \colon u \mapsto q(T, q_0,u)$.  
\end{defi}

\begin{proposition}
Take a reference extremal $z(\cdot) = (q(\cdot),p(\cdot))$ on $[0,T]$, where the corresponding control is given by \eqref{eq:scl_optimal_control}. If we endow the set of controls (valued in $\norme{u}=1$) with the $L^\infty$-norm topology we have:
\begin{enumerate}
\item In the normal case, $u$ is a singularity of the fixed time extremity mapping, that is the image of the Fr\'echet derivative is not of maximal rank.
\item In the abnormal case, $u$ is a singularity of the extremity mapping.
\end{enumerate}
\end{proposition}

\begin{defi}
Let $t \mapsto q(t)$ be a response of \eqref{eq:scl_dynamics}. It is called \emph{regular} if it is a one-to-one immersion. From the Maximum Principle, the geodesics are parameterized by the initial heading angle $\alpha_0$ and fixing $q(0)=q_0$, the \emph{exponential mapping} is $\exp_{q_0,t} \colon \alpha_0 \mapsto \Pi(\exp(t\vv{\Htrue})(q_0,\alpha_0))$ where $\Pi : (q,p) \mapsto q$ is the q-projection. Take a strict normal geodesic $q(\cdot)$, a \emph{conjugate point} along $q(\cdot)$ is a point where the exponential mapping is not an immersion and taking all such geodesics, the set of \emph{first} conjugate points forms the \emph{conjugate locus} $C(q_0)$. The \emph{cut point} along a given geodesic is the point where the geodesic loses optimality and they form the \emph{cut locus} $\Sigma(q_0)$. The \emph{separating line} $L(q_0)$ is the set of points where two minimizing geodesics starting from $q_0$ are intersecting.
\end{defi}

\subsection{Carath\'eodory-Zermelo-Goh transformation and accessibility set}
\label{sec:scl_CZG}

In the historical example \cite{Caratheodory}, Carath\'eodory integrated the dynamics of the heading angle $\alpha$ to parameterize the geodesics. This corresponds to the Goh transformation in optimal control and this will be crucial in our study to set Zermelo navigation problems in the Lie algebraic frame. The introduction of the so-called Goh extension in optimal control leads to set the problem in the context of geometric optimal control and permits to use intrinsic computations with Lie brackets, applying the results from \cite{BonnardKupka:1993}, and this is the first contribution of the paper.

\begin{defi}
Consider the control system \eqref{eq:scl_dynamics}, with $q=(x,y)$ and $\norme{u} = 1$. One can set $u=(\cos\alpha,\sin\alpha)$, $\alpha$ being the heading angle of the ship. Denote $\tilde{q} = (q,\alpha)$, $X(\tilde{q}) = F_0(q) + \cos\alpha\,F_1(q) + \sin\alpha\,F_2(q)$ and $Y(\tilde{q}) = \frac{\partial}{\partial \alpha}$. This leads to prolongate \eqref{eq:scl_dynamics} into the \emph{single-input affine system}: 
\begin{equation}
\label{eq:scl_extend_system}
\dot{\tilde{q}} = X(\tilde{q}) + v\, Y(\tilde{q})
\end{equation}
and the derivative of the heading angle $v=\dot{\alpha}$ is the \emph{accessory control}, $v$ belonging to the whole $\R$.
\end{defi}

We refer to \cite{BonnardKupka:1993} for a presentation of such a transformation in a general context. In this prolongation, the extremal curves $z=(q,p)$ extend to \emph{singular extremal curves} associated to \eqref{eq:scl_extend_system} with coordinates $\tilde{z} = (\tilde{q}, \tilde{p}) = (q,\alpha,p,p_\alpha)$. This leads to define the extended Hamiltonian:
\begin{equation*}
\tilde{H}(\tilde{z}, v) = \tilde{p}\cdot (X(\tilde{q}) + v\, Y(\tilde{q})) + p^0. 
\end{equation*}
From \cite[Chapter~6]{BonnardChyba:2003}, using the Maximum Principle we obtain the following parameterization of the geodesic curves.
%
 Let $\gamma$ be a reference geodesic for the extended system defined on $[0,T]$. We assume the following:
\begin{itemize}
\item[\textbf{(A0)}] The q-projection of $\gamma$ is regular.
\end{itemize}

Computing the Lie bracket with the convention $\liebra{X}{Y}(\tilde{q}) = \frac{\partial X}{\partial \tilde{q}}(\tilde{q}) Y(\tilde{q}) - \frac{\partial Y}{\partial \tilde{q}}(\tilde{q}) X(\tilde{q})$, then straightforward computations give that under assumption (A0) the following holds along $\gamma$.

%
\begin{itemize}
\item[\textbf{(A1)}] $X$, $Y$ are linearly independent.
\item[\textbf{(A2)}] $Y$, $\liebra{X}{Y}$ are linearly independent and the reference geodesic is strict.
\item[\textbf{(A3)}] The strict generalized Legendre-Clebsch condition ($ \frac{\partial}{\partial v} \frac{\diff^2}{\diff t^2} \frac{\partial \tilde{H}}{\partial v} \neq 0$) is satisfied and so we have: $\liebra{\liebra{Y}{X}}{Y} \notin \mathrm{Span}\{Y,\liebra{Y}{X}\}$.
\end{itemize}
\begin{remark}
In a Zermelo navigation problem, the collinearity set of $X$ and $Y$ is: $\norme{F_0(q)}_g=1$.
\end{remark}
Hence from \cite[Chapter~6]{BonnardChyba:2003}, the control $v$ associated to $\gamma$ can be computed as the feedback
\begin{equation*}
v(\tilde{q}) = -\frac{D'(\tilde{q})}{D(\tilde{q})},
\end{equation*}
where we denote
\begin{equation*}
\begin{aligned}
D   = \det( Y, \liebra{Y}{X}, \liebra{\liebra{Y}{X}}{Y}  ) ~~ \text{and} ~~ 
D'  = \det( Y, \liebra{Y}{X}, \liebra{\liebra{Y}{X}}{X}  ).
\end{aligned}
\end{equation*}
Moreover, introducing
$
D'' = \det( Y, \liebra{Y}{X}, X), 
$
we have the following.
\begin{proposition}
\label{prop:scl_extremal_parameterization}
The hyperbolic geodesics are contained in $D D'' > 0$,
the elliptic geodesics are contained in $D D'' < 0$
and the abnormal (or exceptional) geodesics are located in $D'' = 0$.
\end{proposition}

The important application is to use the fine computations detailed in \cite{BonnardKupka:1993}, \cite[Chapter~6]{BonnardChyba:2003} to evaluate in our problem the extremity mapping in the neighborhood of a reference geodesic curve $\gamma$, using \emph{semi-normal forms}, for the action of the feedback group. This gives evaluation of the accessibility sets and \emph{their boundaries}, filled by geodesic curves. 

\subsubsection{Semi-normal forms}

We proceed as follows. For a reference geodesic curve $t \mapsto \gamma(t)$ on $[0,T]$, under the action of the \emph{feedback group}, one can identify $\gamma$ to $t\mapsto (t,0,0)$ and it can be taken as the response of $v\equiv 0$.
Normalization is then obtained in the jet spaces of (X,Y), in the neighborhood of $\gamma$. This is convenient to distinguish normal and abnormal cases.

\paragraph{\textbf{Normal case}} We can choose coordinates $\tilde{q} = (q_1,q_2,q_3)$ such that:
\begin{equation}
\label{eq:scl_forme_normal}
\begin{aligned}
X = \left(1 + \sum_{i,j=2}^{3} a_{i,j}(q_1)q_iq_j\right)\frac{\partial}{\partial q_1} + q_3\frac{\partial}{\partial q_2} + \varepsilon_1, \quad 
Y = \frac{\partial}{\partial q_3},
\end{aligned}
\end{equation}
with $a_{33} <0$ (resp. $a_{33} >0$) in the hyperbolic (resp. elliptic) case.

\paragraph{\textbf{Abnormal case}} We can choose coordinates $\tilde{q} = (q_1,q_2,q_3)$ such that:
\begin{equation}
\label{eq:scl_forme_normal_abnormal}
\begin{aligned}
X = \left(1 + q_2\right)\frac{\partial}{\partial q_1} + \frac{1}{2}a(q_1)q_2^2\frac{\partial}{\partial q_3} + \varepsilon_2, \quad 
Y = \frac{\partial}{\partial q_2}.
\end{aligned}
\end{equation}
Again, see \cite{BonnardKupka:1993} for details of the computations and the descriptions of the mappings $q \mapsto \varepsilon_1(q)$, $\varepsilon_2(q)$. Taking $\varepsilon_i = 0$ and $q_1 = t$ in \eqref{eq:scl_forme_normal} and \eqref{eq:scl_forme_normal_abnormal}, one can evaluate the accessibility set and its boundary and compute conjugate points (in the regular case) to deduce the optimality status of the reference geodesic.

\subsubsection{Optimality status: normal case}

Using the normalization in \eqref{eq:scl_forme_normal} one sets: $q_1(t) = t + w_1(t)$ and projection of the accessibility set in $w_1$-direction is represented on Fig.~\ref{fig:accessibility_scl_1}. Note that hyperbolic and elliptic geodesics amount respectively to minimize and maximize the $w_1$-coordinate. If $t>t_{1c}$ (first conjugate time) the fixed extremity mapping becomes open.

\begin{figure}[ht!]
    \centering
    \begin{tikzpicture}[scale=0.3]
        \def\xh{0.3cm}
        \def\yh{0.3cm}
        \def\ystickz{0.8em}
        \coordinate (O)    at (0,0);
        \coordinate (Ox-)  at (0, 0);
        \coordinate (Ox+)  at (10, 0);
        \coordinate (Oy)   at (0, 10);
        \coordinate (tc)   at (6, 0);
    
        \coordinate (A)   at (8.0, 8.0);
        \coordinate (B)   at (3.0, 3.0);
        \coordinate (C)   at (5.0, 5.0);
        \coordinate (D)   at (7.0, 7.0);
        \coordinate (B1)  at (3.0, 7.5);
        \coordinate (C1)  at (5.0, 9.5);
        \coordinate (D1)  at (7.0, 10.5);
        \coordinate (D2)  at (7.0, 3.5);
    
        \draw (Ox+) node[below right] {$t$};
        \draw (Oy)  node[left] {$w_1$};
        \draw (O)   node[below left] {0};
        
        \draw [->,thin,>=latex] ([xshift=-1cm] Ox-) -- (Ox+);
        \draw [->,thin,>=latex] ([yshift=-1cm] O) --   (Oy);
    
        \draw [color=blue, thick]   (O)  to (A) ;
        \draw [color=black, thick]  (B)  to (B1) ;
        \draw [color=black, thick]  (C)  to (C1) ;
        \draw [color=black, thick]  (D2) to (D1) ;
        \draw [color=black, thick]  (D)  to (8.0, 8.0) ;
        
        \def\dy{1.0cm}
        \draw [black, thick] ([xshift=-\xh, yshift=-\yh+1*\dy] B) -- ([xshift=\xh, yshift=\yh+1*\dy] B);
        \draw [black, thick] ([xshift=-\xh, yshift=-\yh+2*\dy] B) -- ([xshift=\xh, yshift=\yh+2*\dy] B);
        \draw [black, thick] ([xshift=-\xh, yshift=-\yh+3*\dy] B) -- ([xshift=\xh, yshift=\yh+3*\dy] B);
        \draw [black, thick] ([xshift=-\xh, yshift=-\yh+4*\dy] B) -- ([xshift=\xh, yshift=\yh+4*\dy] B);
        
        \def\dy{1.0cm}
        \draw [black, thick] ([xshift=-\xh, yshift=-\yh+1*\dy] C) -- ([xshift=\xh, yshift=\yh+1*\dy] C);
        \draw [black, thick] ([xshift=-\xh, yshift=-\yh+2*\dy] C) -- ([xshift=\xh, yshift=\yh+2*\dy] C);
        \draw [black, thick] ([xshift=-\xh, yshift=-\yh+3*\dy] C) -- ([xshift=\xh, yshift=\yh+3*\dy] C);
        \draw [black, thick] ([xshift=-\xh, yshift=-\yh+4*\dy] C) -- ([xshift=\xh, yshift=\yh+4*\dy] C);
        
        \def\dy{1.0cm}
        \draw [black, thick] ([xshift=-\xh, yshift=-\yh-1*\dy] D) -- ([xshift=\xh, yshift=\yh-1*\dy] D);
        \draw [black, thick] ([xshift=-\xh, yshift=-\yh-2*\dy] D) -- ([xshift=\xh, yshift=\yh-2*\dy] D);
        \draw [black, thick] ([xshift=-\xh, yshift=-\yh-3*\dy] D) -- ([xshift=\xh, yshift=\yh-3*\dy] D);
        \draw [black, thick] ([xshift=-\xh, yshift=-\yh+1*\dy] D) -- ([xshift=\xh, yshift=\yh+1*\dy] D);
        \draw [black, thick] ([xshift=-\xh, yshift=-\yh+2*\dy] D) -- ([xshift=\xh, yshift=\yh+2*\dy] D);
        \draw [black, thick] ([xshift=-\xh, yshift=-\yh+3*\dy] D) -- ([xshift=\xh, yshift=\yh+3*\dy] D);
        

        %
        \shade[ball color=blue] (B) circle (0.2);
        \shade[ball color=blue] (C) circle (0.2);
        \shade[ball color=blue] (D) circle (0.2);
        
        \draw [black] ([yshift=\ystickz] tc) to ([yshift=-\ystickz] tc) ;
        \draw [black, thin, densely dotted] ([yshift=10cm] tc) -- (tc) ;
        \draw (tc)   node[below] {$t_{1c}$};
        
        \draw (5.0,-2.0) node[below]{elliptic case};
            
    \end{tikzpicture}
    \hspace{2em}
    \begin{tikzpicture}[scale=0.3]
        \def\xh{0.3cm}
        \def\yh{0.3cm}
        \def\ystickz{0.8em}
        \coordinate (O)    at (0,0);
        \coordinate (Ox-)  at (0, 0);
        \coordinate (Ox+)  at (10, 0);
        \coordinate (Oy)   at (0, 10);
        \coordinate (tc)   at (6, 0);
    
        \coordinate (A)   at (8.0, 8.0);
        \coordinate (B)   at (3.0, 3.0);
        \coordinate (C)   at (5.0, 5.0);
        \coordinate (D)   at (7.0, 7.0);
        \coordinate (B1)  at (3.0, 0.5);
        \coordinate (C1)  at (5.0, 1.5);
        \coordinate (D1)  at (7.0, 10.5);
        \coordinate (D2)  at (7.0, 3.5);
    
        \draw (Ox+) node[below right] {$t$};
        \draw (Oy)  node[left] {$w_1$};
        \draw (O)   node[below left] {0};
    
        \draw [->,thin,>=latex] ([xshift=-1cm] Ox-) -- (Ox+);
        \draw [->,thin,>=latex] ([yshift=-1cm] O) --   (Oy);
    
        \draw [color=red, thick]   (O) to (A) ;
        \draw [color=black, thick]   (B) to (B1) ;
        \draw [color=black, thick]   (C) to (C1) ;
        \draw [color=black, thick]   (D2) to (D1) ;
        \draw [color=black, thick]   (D) to (8.0, 8.0) ;
        
        \def\dy{1.0cm}
        \draw [black, thick] ([xshift=-\xh, yshift=-\yh-1*\dy] B) -- ([xshift=\xh, yshift=\yh-1*\dy] B);
        \draw [black, thick] ([xshift=-\xh, yshift=-\yh-2*\dy] B) -- ([xshift=\xh, yshift=\yh-2*\dy] B);
        
        \def\dy{1.0cm}
        \draw [black, thick] ([xshift=-\xh, yshift=-\yh-1*\dy] C) -- ([xshift=\xh, yshift=\yh-1*\dy] C);
        \draw [black, thick] ([xshift=-\xh, yshift=-\yh-2*\dy] C) -- ([xshift=\xh, yshift=\yh-2*\dy] C);
        \draw [black, thick] ([xshift=-\xh, yshift=-\yh-3*\dy] C) -- ([xshift=\xh, yshift=\yh-3*\dy] C);
        
        \def\dy{1.0cm}
        \draw [black, thick] ([xshift=-\xh, yshift=-\yh-1*\dy] D) -- ([xshift=\xh, yshift=\yh-1*\dy] D);
        \draw [black, thick] ([xshift=-\xh, yshift=-\yh-2*\dy] D) -- ([xshift=\xh, yshift=\yh-2*\dy] D);
        \draw [black, thick] ([xshift=-\xh, yshift=-\yh-3*\dy] D) -- ([xshift=\xh, yshift=\yh-3*\dy] D);
        \draw [black, thick] ([xshift=-\xh, yshift=-\yh+1*\dy] D) -- ([xshift=\xh, yshift=\yh+1*\dy] D);
        \draw [black, thick] ([xshift=-\xh, yshift=-\yh+2*\dy] D) -- ([xshift=\xh, yshift=\yh+2*\dy] D);
        \draw [black, thick] ([xshift=-\xh, yshift=-\yh+3*\dy] D) -- ([xshift=\xh, yshift=\yh+3*\dy] D);
            
        %
        
        \shade[ball color=red] (B) circle (0.2);
        \shade[ball color=red] (C) circle (0.2);
        \shade[ball color=red] (D) circle (0.2);
        
        \draw [black] ([yshift=\ystickz] tc) to ([yshift=-\ystickz] tc) ;
        \draw [black, thin, densely dotted] ([yshift=10cm] tc) -- (tc) ;
        \draw (tc)   node[below] {$t_{1c}$};
        
        \draw (5.0,-2.0) node[below]{hyperbolic case};
        
\end{tikzpicture}
\vspace{-0.5cm}
\caption{Projection of the fixed time accessibility set on the $w_1$-coordinate.
}
\vspace{-0.3cm}
\label{fig:accessibility_scl_1}
\end{figure}
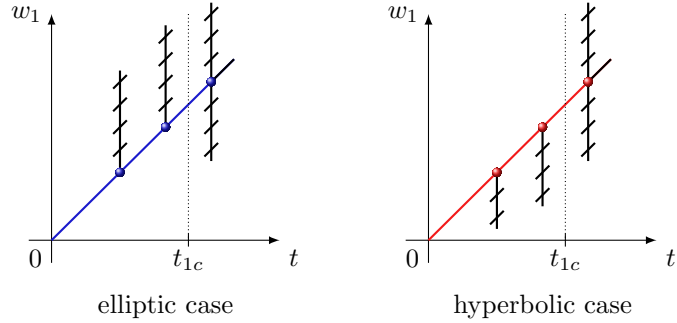

\begin{proposition}
Under the assumptions, in the hyperbolic (resp. elliptic) case, the reference geodesic $\gamma$ is time minimizing (resp. maximizing) with respect to all trajectories of the system, contained in a conic neighborhood of the reference curve, if the final time is strictly less than the first conjugate time $t_{1c}$.
\end{proposition}

\subsubsection{Optimality status: abnormal case} 
\label{sec:optimality_abnormal}

In this case, one must estimate the \emph{time evolution} of the accessibility set and its boundary. It is represented on Fig.~\ref{fig:accessibility_scl_3}. The reference geodesic is $\gamma \colon t \mapsto (t,0,0)$ and is associated to $v\equiv 0$. We fix $t$ along the reference curve and let a time $t_f$ in a neighborhood of $t$. Using the model, we compute geodesics such that:
\begin{equation*}
    q_1(t_f) = t, \quad q_2(t_f) = 0,
\end{equation*}
and the associated cost is given by
%
    $q_3(t_f) = \frac{1}{2} \int_0^{t_{f}} a(q_1)q_2^2\, \diff t.$
%
Note that if we restrict to geodesics, this amounts to use the Jacobi (variational) equation, along the reference geodesic. One has:
\begin{equation}
\label{eq:scl_abnormal_form}
    q_3(t_f) = \alpha(t-t_f)^2 + o(t-t_f)^3,
\end{equation}
$\alpha$ being a positive invariant, given by the Jacobi equation.

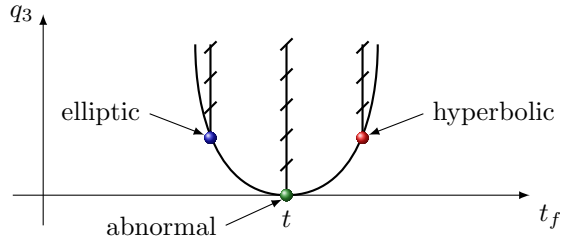
\begin{figure}[ht!]
\centering
       \begin{tikzpicture}[scale=0.4]
       
            \def\ystickz{0.5em}
       
            \coordinate (O)    at (0 , 0);
            \coordinate (Ox-)  at (-9, 0);
            \coordinate (Ox+)  at ( 8, 0);
            \coordinate (Oy)   at ( 0, 5);
            \coordinate (tc)   at ( 5, 0);
        
            \coordinate (O1)   at ( -8, -1);
            \coordinate (O2)   at ( -8,  6);
        
            \coordinate (A)   at (-3.0, 5.0);
            \coordinate (B)   at ( 3.0, 5.0);
            \coordinate (C)   at (-2.5, 2.0);
            \coordinate (D)   at ( 2.5, 2.0);
            \coordinate (C1)  at (-2.5, 5.2);
            \coordinate (C2)  at ( 2.5, 5.2);
        
            \draw (Ox+) node[below right] {$t_f$};
            \draw (O2)  node[left] {$q_3$};
        
            \draw [->,thin,>=latex] (Ox-) -- (Ox+);
            \draw [->,thin,>=latex] (O1) -- (O2);
        
            \draw [color=black, thick] (A) to[out=270,in=180] (O) ;
            \draw [color=black, thick] (O) to[out=0,in=270] (B) ;
            
            \def\xh{0.2cm}
            \def\yh{0.2cm}
        
            \coordinate (E)  at (-2.5,1.9);
            \coordinate (E1) at (-2.5,5.0);
            \draw [color=black, thick] (E) to (E1) ;
            \def\dy{1.0cm}
            \draw [black, thick] ([xshift=-\xh, yshift=-\yh+1*\dy] E) -- ([xshift=\xh, yshift=\yh+1*\dy] E);
            \draw [black, thick] ([xshift=-\xh, yshift=-\yh+2*\dy] E) -- ([xshift=\xh, yshift=\yh+2*\dy] E);
            \draw [black, thick] ([xshift=-\xh, yshift=-\yh+3*\dy] E) -- ([xshift=\xh, yshift=\yh+3*\dy] E);
            
            \coordinate (H)  at (2.5,1.9);
            \coordinate (H1) at (2.5,5.0);
            \draw [color=black, thick] (H) to (H1) ;
            \draw [black, thick] ([xshift=-\xh, yshift=-\yh+1*\dy] H) -- ([xshift=\xh, yshift=\yh+1*\dy] H);
            \draw [black, thick] ([xshift=-\xh, yshift=-\yh+2*\dy] H) -- ([xshift=\xh, yshift=\yh+2*\dy] H);
            \draw [black, thick] ([xshift=-\xh, yshift=-\yh+3*\dy] H) -- ([xshift=\xh, yshift=\yh+3*\dy] H);
            
            \coordinate (Ab)  at (0.0,0.0);
            \coordinate (Ab1) at (0.0,5.0);
            \draw [color=black, thick] (Ab) to (Ab1) ;
            \draw [black, thick] ([xshift=-\xh, yshift=-\yh+1*\dy] Ab) -- ([xshift=\xh, yshift=\yh+1*\dy] Ab);
            \draw [black, thick] ([xshift=-\xh, yshift=-\yh+2*\dy] Ab) -- ([xshift=\xh, yshift=\yh+2*\dy] Ab);
            \draw [black, thick] ([xshift=-\xh, yshift=-\yh+3*\dy] Ab) -- ([xshift=\xh, yshift=\yh+3*\dy] Ab);
            \draw [black, thick] ([xshift=-\xh, yshift=-\yh+4*\dy] Ab) -- ([xshift=\xh, yshift=\yh+4*\dy] Ab);
            \draw [black, thick] ([xshift=-\xh, yshift=-\yh+5*\dy] Ab) -- ([xshift=\xh, yshift=\yh+5*\dy] Ab);
            
            \draw[->,>=latex] (-2.0, -1.0) to ([xshift=-0.4em, yshift=-0.4em] O);
            \draw (-2.0, -1.0) node[black, left]{abnormal};
            
            \draw[->,>=latex] (4.5, 2.7) to (2.7,2.0);
            \draw (4.5, 2.7) node[black, right]{hyperbolic};
            
            \draw[->,>=latex] (-4.5, 2.7) to (-2.7,2.0);
            \draw (-4.5, 2.7) node[black, left]{elliptic};
            
            \shade[ball color=green] (Ab) circle (0.2);
            \shade[ball color=blue] (E) circle (0.2);
            \shade[ball color=red] (H) circle (0.2);
            
            \draw (-0.0,-0.2) node[below]{$t$};
            
            
            
\end{tikzpicture}
\vspace{-0.3cm}
\caption{Projection of the accessibility sets on the $q_3$-coordinate in the abnormal case.}
    \label{fig:accessibility_scl_3}
\end{figure}

Note that the model \eqref{eq:scl_forme_normal_abnormal} shows clearly that the abnormal curve is a limit curve, as observed by Carath\'eodory. In the n-dimensional case, conjugate points along the abnormal curve can be computed and correspond to points where the extremity mapping becomes open. But clearly from \eqref{eq:scl_abnormal_form} this cannot occur in the 3d-case.
In particular one has:

\begin{proposition}
Under our assumptions, in the abnormal (exceptional) case the reference geodesic is time minimizing and time maximizing, with respect to all trajectories contained in a conic neighborhood of the reference curve.
\end{proposition}
%

\subsection{Small-time balls and spheres in the strong current case}

One consequence of our previous analysis is to recover the fan shape of the small-time balls in the strong current case, described in the historical example. This is the second contribution of this article and we proceed as follows.

\subsubsection{The tangent model}

For the illustration, we consider that $\norme{F_0(q_0)}_g>1$ with $F_0(q_0)$ taken horizontal and pointing in the right direction and the metric $g$ given by $F_1(q_0)=\frac{\partial}{\partial x}$, $F_2(q_0)=\frac{\partial}{\partial y}$ with $q = (x,y)$. 
The ball of directions at $q_0$ is defined by:
\[
F(q_0) = \enstq{F_0(q_0) + u}{\norme{u}\leq 1}.
\]
It is represented on Fig.~3 and its boundary is a translation of the unit circle. The two abnormal directions are associated to the heading angles denoted $\{-\alpha^a, \alpha^a \}$ and correspond to the \emph{tangents} of the translated unit circle from $q_0$. These heading angles split the translated unit circle in two, the right part corresponds to hyperbolic directions and the left part to elliptic directions.
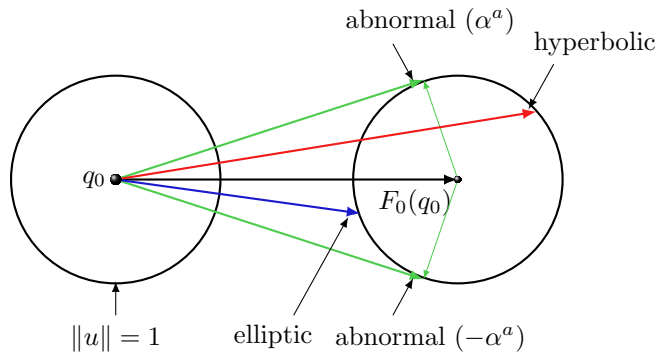
\begin{figure}[ht!]
\def\scaling{0.9}
\centering
\begin{tikzpicture}[scale=\scaling, rotate =-90]
\centering
    \pgfmathsetmacro{\Rinitial}{0.08}
    \pgfmathsetmacro{\Rorbite}{0.4}
    \pgfmathsetmacro{\R}{1.53}
    \pgfmathsetmacro{\Rpoint}{0.07} node[yshift=1];

    \coordinate (q0) at (0,-5.0);
    \draw[color=black,thick] (0,0) circle (\R);
    \draw[color=black,thick] (q0) circle (\R);
    
    \coordinate (x0) at ( 4,0.3);
    \coordinate (xf) at (-4,1);
    \coordinate (xi) at ( 2,0);
    \coordinate (xo) at (-1.9021,0.6180);

    \coordinate(A1) at ( 1.47,-0.5);
    \coordinate(A2) at (-1.47,-0.5);
    \coordinate(P1) at ( 1.0, 1.15);
    \coordinate(P2) at (-1.0, 1.15);
    \coordinate(P3) at ( 0.5,-1.42);

    \draw[color=black, ->,>=latex, thick] (q0) to (0,0);
    \draw (0.0,-1.3) node[black, below right]{$F_0(q_0)$};
    
    \draw[color=greenDrift, ->,>=latex,thick] (q0) to (A1);
    \draw[color=greenDrift, ->,>=latex, thick] (q0) to (A2);
    \draw[color=greenDrift, ->,>=latex, very thin] (0,0) to (A1);
    \draw[color=greenDrift, ->,>=latex, very thin] (0,0) to (A2);
    
    \draw[color=red, ->,>=latex, thick] (q0) to (P2);
    
    \draw[color=blue, ->,>=latex, thick] (q0) to (P3);

    
    \draw[->,>=latex,thin] (-2., -0.95) to (-1.45,-0.65);
    \draw (-2.,-0.4) node[black, above]{abnormal ($\alpha^a$)};

	\draw[->,>=latex,thin] (2., -0.95) to (1.45,-0.65);
    \draw (2.,-0.4) node[black, below]{abnormal ($-\alpha^a$)};
    
    \draw[->,>=latex,thin] (-1.7, 1.4) to (-1.03, 1.03);
    \draw (-1.7, 2.) node[black, above]{hyperbolic};
    
    \draw[->,>=latex,thin] (2., -2.4) to (0.55,-1.6);
    \draw (2., -2.7) node[black, below]{elliptic};
    
     \draw[->,>=latex,thin] (2.0,-5.0) to (1.5,-5.0);
     \draw (2.0,-5.0) node[black, below]{$\norme{u}=1$};
    
     \node[left] at (q0) {$q_0$};

    \shade[ball color=black] (q0) circle (\Rinitial);
    \shade[ball color=black] (0,0) circle (0.05);
\end{tikzpicture} 
\vspace{-0.25cm}
\caption{Strong current case: ball of directions.}
\label{fig:strong_drift_scl}
\end{figure}
%

\subsubsection{Small spheres and balls}

Using Section \ref{sec:scl_CZG}, one gets the following.

\begin{proposition}
In the strong current case for a small-time, the exponential mapping is a diffeomorphism from the unit circle onto its image, which is formed on one part by the extremities of the hyperbolic trajectories, the other part being the extremities of the elliptic trajectories, the two parts being separated by the two points corresponding to the abnormal trajectories. Hyperbolic and elliptic geodesics correspond respectively to time minimizing and time maximizing trajectories, while abnormal geodesics are time minimizing and time maximizing. 
\end{proposition}

The sphere and the ball with small radius are represented on Fig.~\ref{fig:sphere_strong_scl} and in particular this gives the fan shape of the corresponding balls. The contact of the hyperbolic sector with the abnormal curve can be obtained as in section \ref{sec:optimality_abnormal} using the micro-local model where the abnormal geodesic is normalized to the horizontal line. A more precised representation can be obtained in the rotational case, since the geodesic flow is Liouville integrable, which leads to the exact computation of the exponential mapping. This point of view will be developed in the next section.

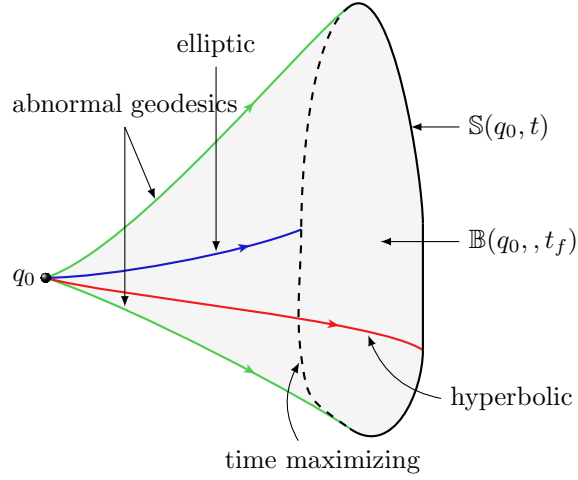
\begin{figure}[ht!]
\def\scaling{0.8}
\centering
    \begin{tikzpicture}[scale=\scaling, rotate=-90]
        \coordinate (q0) at (0,-4.0);

        \coordinate (x0) at ( 4,0.3);
        \coordinate (xf) at (-4,1);
        \coordinate (xi) at ( 2,0);
        \coordinate (xo) at (-1.9021,0.6180);

        \coordinate(A1) at (-4.5, 1.0); 
        \coordinate(A2) at ( 2.5, 1.0); 

        \coordinate(A11)  at ( -0.9, 2.2);
        \coordinate(A12)  at (  1.2, 2.2);

        \coordinate(P1) at (-0.8, 0.2);
        \coordinate(P2) at ( 1.3, 0.2);

        \def\pxa{0.7} 
        \def\pya{0.0} 
        \def\pxb{0.7} 
        \def\pyb{0.0} 
        \filldraw[draw=none,fill=gray!20,opacity=0.4] 
        (q0) .. controls +(-0.5,1.5) and +( 0.5,-0.71) .. (A1)
        .. controls +(-0.5, 0.71)    and +(-\pxa, \pya) .. (A11)
        .. controls +( \pxa, -\pya)  and +(-\pxb, \pyb) .. (A12)
        .. controls +( \pxb, -\pyb)  and +(  0.5, 0.71) .. (A2)
        .. controls +(-0.5,-0.71) and +( 0.5,1.5) .. (q0)
        ;

        \draw [thick, greenDrift] (q0) .. controls +(-0.5,1.5) and +( 0.5,-0.71) .. (A1)
        node[pos=0.6, sloped, rotate=90] {\tiny \ding{228}}; 

        \draw [thick, greenDrift] (q0) .. controls +( 0.5,1.5) and +(-0.5,-0.71) .. (A2)
            node[pos=0.6, sloped, rotate=-90] {\tiny \ding{228}}; 

        \draw [thick, red] (q0) .. controls +(0.3,1.0) and +(-0.5,-0.71) .. (A12)
        node[pos=0.7, sloped, rotate=-90] {\tiny \ding{228}};

        \draw [thick, blue ] (q0) .. controls +(0.0,1.0) and +( 0.4,-1.0) .. (P1)
        node[pos=0.75, sloped, rotate=90] {\tiny \ding{228}};

        \draw [thick, black] (A1)  .. controls +(-0.5, 0.71)    and +(-\pxa, \pya) .. (A11);
        \draw [thick, black] (A11) .. controls +(\pxa, -\pya)  and +(-\pxb, \pyb) .. (A12);
        \draw [thick, black] (A12) .. controls +(\pxb, -\pyb)  and +(0.5, 0.71) .. (A2);

        \def\pxa{0.7}
        \def\pya{0.0}
        \def\pxb{0.7}
        \def\pyb{-0.1}
        \draw [thick, black, dashed] (A1) .. controls +( 0.5, -0.71)   and +(-\pxa, \pya) .. (P1);
        \draw [thick, black, dashed] (P1) .. controls +( \pxa, -\pya)  and +(-\pxb, \pyb) .. (P2);
        \draw [thick, black, dashed] (P2) .. controls +( \pxb, -\pyb)  and +(-0.5, -0.71) .. (A2);

    \draw[->,>=latex] (-2.5, -2.7) to (0.48, -2.7);
    \draw[->,>=latex] (-2.5, -2.7) to (-1.3, -2.2);
    \draw (-2.5, -2.7) node[black, above]{abnormal geodesics};

    \draw[->,>=latex, bend left] (2., 2.5) to (0.95, 1.3);
    \draw (2., 2.5) node[black, right]{hyperbolic};

    \draw[->,>=latex] (-3.5,-1.2) to (-0.43,-1.2);
    \draw (-3.5,-1.2) node[black, above]{elliptic};

    \draw[->,>=latex] (-2.5, 2.8) to (-2.5, 2.);
    \draw (-2.5, 2.8) node[black, right]{$\MSphere(q_0,t)$};
    
    \draw[->,>=latex] (-0.6, 2.8) to (-0.6, 1.5);
    \draw (-0.6, 2.8) node[black, right]{$\MBall(q_0,,t_f)$};    

    \draw[->,>=latex, bend left] (2.7, 0.15) to (1.4, 0.17);
    \draw (2.7, 0.55) node[black, below]{time maximizing};
    
    \shade[ball color=black] (q0) circle (0.08);
    \node[left] at (q0) {$q_0$};
    
\end{tikzpicture}
\vspace{-1.em}
\caption{Small sphere (plain black line) and ball (delimited by the abnormal geodesics and the sphere) for the strong current case. Note that the elliptic geodesics do not play a role in the construction of the time minimal sphere. But the part of the wavefront formed by the elliptic geodesics is represented (dashed line), together with an elliptic geodesic, to emphasize the difference with the weak case and represent the accessibility set in time $t_f$ which is the interior of the wavefront.}
\label{fig:sphere_strong_scl}
\end{figure}
%

\section{The cusp singularity in the abnormal direction and regularity of the time minimal value function}
\label{sec:scl_cusp}

The main contribution of this article is to analyze, based on the historical example from Carath\'eodory and Zermelo, the deformation (morphogenesis) of the boundary of the accessibility set when meeting the transition between the domain of strong current and the domain of weak current. This causes a cusp singularity of the abnormal extremal which is in particular at such point not an immersed curve\footnote{This is not covered by the article \cite{BonnardKupka:1993}.}. We refer to \cite{BCW2019, Caratheodory} to the occurrence of the cusp singularity in both examples, which motivates the study of this \emph{stable} singularity. This phenomenon is analyzed in this article in relation with the non-continuity of the value function in the strong current Zermelo case and the appearance of a new branch in the cut locus not present in the Finsler case. In the following, the problem is set as a family of problems with rotational symmetry, which covers the two case studies. Note that the historical example is taken as a semi-normal form in the case of revolution, and we refer to \cite{BRW:2021} for the analysis in the general case.

\subsection{The geometric frame and integrability properties}

Recall that in Dardoux coordinates $(r,\theta)$, we consider a metric of the form $g = \diff r^2 + m^2(r) \diff\theta^2$ and a current $F_0(q) = \mu(r)\frac{\partial}{\partial \theta}$. With such a metric, $F_1 = \frac{\partial}{\partial r}$ and $F_2 = \frac{1}{m(r)}\frac{\partial}{\partial \theta}$
form an orthonormal frame.
Using the Carath\'eodory-Zermelo-Goh extension one gets with $\tilde{q} = (r,\theta, \alpha)$ ($\alpha$ being the heading angle):
\[
X = \cos\alpha\frac{\partial}{\partial r} + \left(\mu(r) + \frac{\sin\alpha}{m(r)}\right) \frac{\partial}{\partial \theta}.
\]
Straightforward computations give: 
\[
\begin{aligned}
\liebra{Y}{X}(\tilde{q}) \hspace{1.8em} &= \sin{\alpha}\frac{\partial}{\partial r} - \frac{\cos{\alpha}}{m(r)}\frac{\partial}{\partial \theta}, \\ 
\liebra{\liebra{Y}{X}}{Y}(\tilde{q}) &= \cos{\alpha}\frac{\partial}{\partial r} + \frac{\sin{\alpha}}{m(r)}\frac{\partial}{\partial \theta}, \\ 
\liebra{\liebra{Y}{X}}{X}(\tilde{q}) &= \left(-\mu'(r)\sin\alpha +\frac{m'(r)}{m^2(r)}\right) \frac{\partial}{\partial \theta}.
\end{aligned}
\]
Hence, we have:
\[
\begin{aligned}
D(\tilde{q})   =  \frac{1}{m(r)}, ~
D'(\tilde{q})  = -\mu'(r)\sin^2\alpha + \frac{m'(r)\sin\alpha}{m^2(r)}, ~
D''(\tilde{q}) =  \mu(r)\sin{\alpha} + \frac{1}{m(r)}.
\end{aligned}
\]
So that conditions $(A2)$ and $(A3)$ are satisfied, in particular every geodesic is strict.
But the collinearity condition $(A1)$ can be violated at points where 
\begin{equation*}
    \cos\alpha = \mu(r) + \frac{\sin\alpha}{m(r)} = 0.
\end{equation*}
The dynamics is given by
\begin{equation}
\begin{aligned}
    \dot{r}      &= \cos{\alpha}, \\
    \dot{\theta} &= \mu(r) + \frac{\sin\alpha}{m(r)}, \\
    \dot{\alpha} &= \mu'(r)m(r)\sin^2\alpha - \frac{m'(r)\sin\alpha}{m(r)}.
\end{aligned}
 \label{eq:scl_dynamics_in_dimension_3}
\end{equation}
The following is useful.

\begin{proposition}
The dynamics \eqref{eq:scl_dynamics_in_dimension_3} can be integrated by quadrature.
\end{proposition}
\begin{proof}
The pseudo-Hamiltonian takes the form:
\begin{equation}
\label{eq:scl_pseudo_Hamiltonian}
H = p_r\cos\alpha + p_\theta \left(\mu(r) + \frac{\sin\alpha}{m(r)}\right) + p^0.
\end{equation}
Moreover, from the maximization condition one has 
$\frac{\partial H}{\partial \alpha} = 0,$
which gives the \emph{Clairaut relation}
%
$p_r\sin\alpha =  \frac{p_\theta}{m(r)}\cos\alpha.$
%
So, $(p_r,p_\theta/m(r))$ and $(\cos\alpha,\sin\alpha)$ are collinear and one has $(p_r,p_\theta/m(r)) = \lambda (\cos\alpha,\sin\alpha)$ with $\lambda = (p_r^2 + \frac{p_\theta^2}{m^2(r)})^{1/2}$.
Plugging such $p_r$ into \eqref{eq:scl_pseudo_Hamiltonian} allows us to define the following implicit relation between $\alpha$ and $r$:
\begin{equation}
\label{eq:scl_r_implicit_r}
    p_\theta \left(\mu(r) + \frac{1}{m(r)\sin{\alpha}} \right) + p^0 = 0,
\end{equation}
for $\alpha \neq  0~[\pi]$.
In the case where $\alpha_0 = 0~[\pi]$, one has:
\[
\alpha(t) = \alpha_0, ~~ r(t) = \pm t + r_0, ~~ \text{and} ~~ \theta(t) = \int_0^t \mu(r) \diff t. 
\]

Besides, by homogeneity one can fix $\lambda(0)=1$ and so $(p_{r_0},p_\theta/m(r_0)) = (\cos\alpha_0,\sin\alpha_0)$. Then, one gets $p_\theta = m(r_0)\sin\alpha_0$ and from \eqref{eq:scl_r_implicit_r} one deduces $p^0 = -1 - p_\theta \mu(r_0)$.
Equation \eqref{eq:scl_dynamics_in_dimension_3} can now be solved by quadrature and
from geometric control point of view, it amounts to compute first the control using the integration of the heading angle, $r$ being given by equation \eqref{eq:scl_r_implicit_r}. Then, $\theta$ can be obtained using a further quadrature.
\end{proof}
This result can be applied to our two case studies, to give a model of the cusp singularity.
We shall present it in our frame, with the historical model.

\subsection{Computations in the historical example}

We consider now the historical example presented in the introduction. 
We consider the following coordinates $\tilde{q} = (x,y,\gamma) = (\theta, r, \pi/2-\alpha)$, where $r$, $\theta$ and $\alpha$ are understood in the sense of the previous section. In this historical example the functions $\mu(\cdot)$ and $m(\cdot)$ are given by $\mu(y) = y$ and $m(y) = 1$. In this representation, the dynamics takes the following form:
\begin{equation}
    \dot{x} = y + \cos{\gamma},     \quad
    \dot{y} = \sin{\gamma}, \quad
    \dot{\gamma} =-\cos^2\gamma.
    \label{eq:scl_historical_example}
\end{equation}

Straightforward computations using the previous section leads to
\[
D(\tilde{q})   = 1,~~ 
D'(\tilde{q})  = \cos^2\gamma ~~\text{and}~~
D''(\tilde{q}) = y\cos{\gamma} + 1,
\]
and thanks to Proposition~\ref{prop:scl_extremal_parameterization} we can parameterize abnormal, hyperbolic and elliptic extremals.
\begin{itemize}
\item \textbf{Abnormal case.} The abnormal geodesics are contained in $D'' = y\cos{\gamma} + 1 = 0$. Hence, given an initial condition $(x_0,y_0,\gamma_0)$ such that $|y_0| \geq 1$, the associated geodesic is abnormal if $\gamma_0 \in \{\gamma_a^1,\gamma_a^2\}$ with 
\[
\gamma_a^1 = \arccos\left({-\frac{1}{y_0}}\right) \quad \text{and} \quad \gamma_a^2 = -\arccos\left({-\frac{1}{y_0}}\right).
\]
If the current is strong, that is if $|y_0| >1$, then $\gamma_a^1 \neq \gamma_a^2$ and we have two abnormal geodesics. Else, if $|y_0| =1$ there is only one abnormal, and if $|y_0| < 1$ (this corresponds to a weak current) there is no abnormal geodesics.
\item \textbf{Normal case.} The hyperbolic (resp. elliptic) geodesics are contained in $DD'' = D'' >0$ (resp. $DD'' = D'' <0$). Hence, given $(x_0,y_0,\gamma_0)$:
\begin{itemize}
    \item if $|y_0| <1$, then $y_0\cos\gamma_0 + 1 > 0$ and thus the geodesic is hyperbolic.
    \item for $|y_0| = 1$, if the geodesic is normal, then it is hyperbolic.
    \item for $|y_0| > 1$, if the geodesic is normal, then it is either hyperbolic or elliptic depending on the sign of $y_0\cos\gamma_0 + 1$. Note that the hyperbolic and elliptic geodesics are separated by the abnormal geodesics as illustrated in Fig.~\ref{fig:strong_drift_scl}.
\end{itemize}
\end{itemize}

To complete the discussion about the historical example, we give the integration of the system.

\begin{proposition}
Let $(x_0,y_0,\gamma_0)$ be the initial condition, the corresponding solution $(x(t),y(t),\gamma(t))$ is given as follows.

\begin{itemize}
\item For  $\gamma_0 = \pm\pi/2$ one has: 
\[
\gamma(t) = \gamma_0, ~~ y(t) = \pm t + y_0 ~~ \text{and} ~~ x(t) = \pm \frac{t^2}{2} + y_0 t + x_0.
\]
\item For  $\gamma_0 \in (-\pi/2,\pi/2)$, one has:
\begin{equation*}
\begin{aligned}
\gamma(t) &= \atan{(\tan\gamma_0-t)}, ~~~~
y(t)      = y_0 + \frac{1}{\cos\gamma_0} - \frac{1}{\cos{\gamma(t)}}, \\ 
x(t) &=  \frac{1}{2} 
\left[ \ln\left|\frac{\cos{\gamma}}{1+\sin{\gamma}}\right| \right]_{\gamma_0}^{\gamma(t)} +
\frac{1}{2} \left[ \frac{\tan{\gamma}}{\cos{\gamma}} \right]_{\gamma_0}^{\gamma(t)} 
+ \left(y_0 + \frac{1}{\cos\gamma_0} \right) t + x_0.
\end{aligned}
\end{equation*}
\item For  $\gamma_0 \in (-\pi,-\pi/2) \,\cup \,(\pi/2,\pi]$, one has:
\begin{equation*}
\begin{aligned}
\gamma(t) &= \pi + \atan{(\tan\gamma_0-t)}, ~~~~
y(t)      = y_0 + \frac{1}{\cos\gamma_0} - \frac{1}{\cos{\gamma(t)}}, \\ 
x(t) &=  \frac{1}{2} 
\left[ \ln\left|\frac{\cos{\gamma}}{1+\sin{\gamma}}\right| \right]_{\gamma_0}^{\gamma(t)} +
\frac{1}{2} \left[\frac{\tan{\gamma}}{\cos{\gamma}} \right]_{\gamma_0}^{\gamma(t)} 
+ \left(y_0 + \frac{1}{\cos\gamma_0} \right) t + x_0.
\end{aligned}
\end{equation*} 
\end{itemize}
\end{proposition}
\medskip
\paragraph{\textbf{Cusp singularity and regularity of the value function}}
The integration of the system allows us to compute the
cusp points in the abnormal directions.
A cusp point denoted $(x_\mathrm{cusp},y_\mathrm{cusp},\gamma_\mathrm{cusp})$ occurs along an abnormal geodesic at time $t_\mathrm{cusp}$ when 
$\dot{x}(t_\mathrm{cusp}) = \dot{y}(t_\mathrm{cusp}) = 0$.
This gives      
\[
t_\mathrm{cusp} = \tan\gamma_0, ~~ 
\gamma_\mathrm{cusp} = 0~[\pi] ~~\text{and} ~~
y_\mathrm{cusp} = \sign  \left(y_0\right).
\]
Finally, $x_\mathrm{cusp}$ is deduced from the analytical expressions given above.
The geometric features of the model are the following.
\begin{itemize}
\item The abnormal with the cusp singularity is the limit curve of the micro-local sector, formed by self-intersecting hyperbolic geodesics, see Fig.~\ref{fig:flot_geodesics_scl}.
\item The abnormal is optimal up to the cusp point. Hence, it corresponds to a concept of conjugate point along the \emph{non-smooth} abnormal geodesic image, the geodesic being not an immersed curve.
\item Carath\'eodory \cite{Caratheodory} already described the following phenomenon. Due to the loss of local accessibility associated to the limit geodesic, the time minimal value function is \emph{not continuous}. This is clear from Fig.~\ref{fig:cusp_continuity_value_function_scl}. To reach from the initial point $q_0$ to a point $B$ at right of the limit curve, one must use a self-intersecting normal geodesic so that at the intersection with the abnormal geodesic, the time is longer along the normal than along the abnormal geodesic. We also observe that in this sector, the normal geodesic is optimal up to the intersection point with the abnormal geodesic. See \cite{BRW:2021} for the proof of this result in the general case.
\end{itemize}
\def\sizeFig{0.4}
\begin{figure}[ht!]
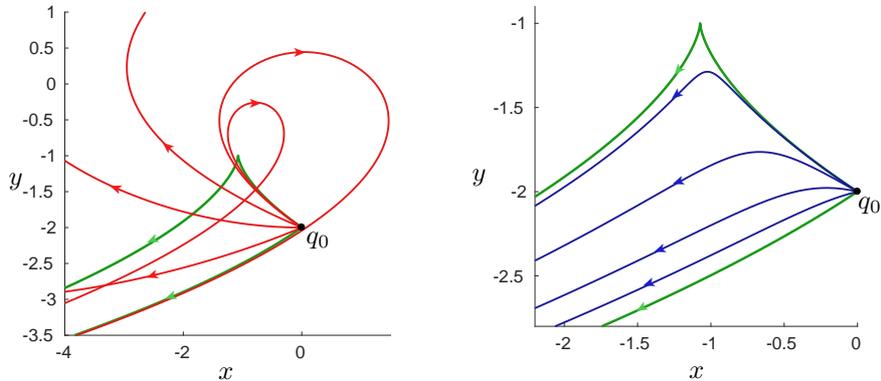

    \centering
    \def\x{493}
    \def\y{470}
    \begin{tikzgraphics}{\sizeFig\textwidth}{\x}{\y}{Flot_hyperboliques}
        \pxcoordinate{0.55*\x}{1.05*\y}{A}; \draw (A) node {$x$};
        \pxcoordinate{-0.02*\x}{0.5*\y}{A}; \draw (A) node {$y$};
        
        \pxcoordinate{0.8*\x}{0.67*\y}{A}; \draw (A) node {$q_0$};
        \pxcoordinate{0.35*\x}{0.67*\y}{A}; \draw (A) node[greenDrift, sloped, rotate=205] {\tiny \ding{228}};
        \pxcoordinate{0.4*\x}{0.83*\y}{A}; \draw (A) node[greenDrift, sloped, rotate=200] {\tiny \ding{228}};
        \pxcoordinate{0.35*\x}{0.77*\y}{A}; \draw (A) node[red, sloped, rotate=200] {\tiny \ding{228}};
        \pxcoordinate{0.25*\x}{0.52*\y}{A}; \draw (A) node[red, sloped, rotate=160] {\tiny \ding{228}};
        \pxcoordinate{0.39*\x}{0.4*\y}{A}; \draw (A) node[red, sloped, rotate=140] {\tiny \ding{228}};
        \pxcoordinate{0.75*\x}{0.13*\y}{A}; \draw (A) node[red, sloped, rotate= 0] {\tiny \ding{228}};
        \pxcoordinate{0.63*\x}{0.275*\y}{A}; \draw (A) node[red, sloped, rotate= 0] {\tiny \ding{228}};
        
    \end{tikzgraphics}
    \hspace{1em}
    \def\sizeFig{0.4}
    \def\x{484}
    \def\y{471}
    \begin{tikzgraphics}{\sizeFig\textwidth}{\x}{\y}{Flot_elliptiques}
        \pxcoordinate{0.55*\x}{1.05*\y}{A}; \draw (A) node {$x$};
        \pxcoordinate{-0.04*\x}{0.5*\y}{A}; \draw (A) node {$y$};
        \pxcoordinate{1.02*\x}{0.58*\y}{A}; \draw (A) node {$q_0$};
        
        \pxcoordinate{0.5*\x}{0.2*\y}{A}; \draw (A) node[greenDrift, sloped, rotate=230] {\tiny \ding{228}};
        \pxcoordinate{0.4*\x}{0.87*\y}{A}; \draw (A) node[greenDrift, sloped, rotate=210] {\tiny \ding{228}};
        \pxcoordinate{0.42*\x}{0.8*\y}{A}; \draw (A) node[blue, sloped, rotate=200] {\tiny \ding{228}};
        \pxcoordinate{0.45*\x}{0.705*\y}{A}; \draw (A) node[blue, sloped, rotate=200] {\tiny \ding{228}};
        \pxcoordinate{0.5*\x}{0.515*\y}{A}; \draw (A) node[blue, sloped, rotate= 200] {\tiny \ding{228}};
        \pxcoordinate{0.495*\x}{0.27*\y}{A}; \draw (A) node[blue, sloped, rotate= 220] {\tiny \ding{228}};
\end{tikzgraphics}
\vspace{-0.5cm}
\caption{(Left) Hyperbolic geodesics (in red) that started from the initial point $q_0 = (-2,0)$ portrayed in black, in the whole conic neighborhood delimited by the two abnormal geodesics (in green). (Right) Elliptic geodesics (in blue) from the same initial point and with the same sector. The set $\norme{F_0(q)}_g=1$ is given by $y=\pm 1$.}
\label{fig:flot_geodesics_scl}
\end{figure}

\def\sizeFig{0.38}
\begin{figure}[ht!]
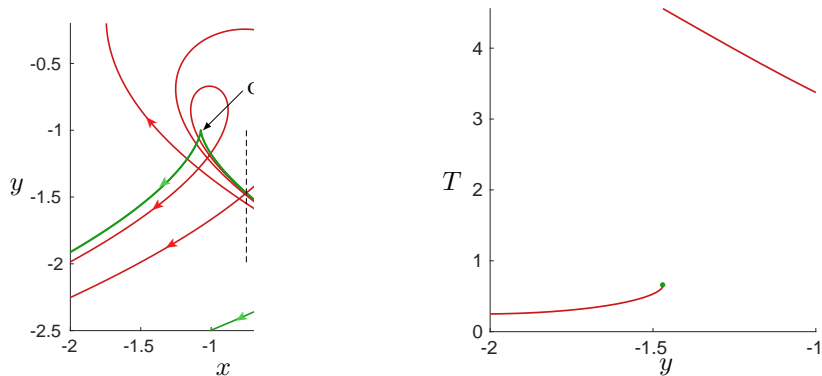

\centering
\def\x{493}
\def\y{470}
\begin{tikzgraphics}{\sizeFig\textwidth}{\x}{\y}{fig_value_function1}
    \pxcoordinate{0.55*\x}{1.05*\y}{A}; \draw (A) node {$x$};
    \pxcoordinate{-0.04*\x}{0.5*\y}{A}; \draw (A) node {$y$};
    
    \pxcoordinate{0.38*\x}{0.49*\y}{A}; \draw (A) node[greenDrift, sloped, rotate=220] {\tiny \ding{228}};
    \pxcoordinate{0.6*\x}{0.895*\y}{A}; \draw (A) node[greenDrift, sloped, rotate=210] {\tiny \ding{228}};
    \pxcoordinate{0.4*\x}{0.675*\y}{A}; \draw (A) node[red, sloped, rotate=210] {\tiny \ding{228}};
    \pxcoordinate{0.34*\x}{0.3*\y}{A}; \draw (A) node[red, sloped, rotate=130] {\tiny \ding{228}};
    \pxcoordinate{0.36*\x}{0.558*\y}{A}; \draw (A) node[red, sloped, rotate=220] {\tiny \ding{228}};
    
    \pxcoordinate{0.97*\x}{0.76*\y}{A}; \draw (A) node {$q_0$};
    \pxcoordinate{0.704*\x}{0.43*\y}{A}; \draw (A) node {{\tiny \textbullet}};
    \pxcoordinate{0.74*\x}{0.46*\y}{A}; \draw (A) node {$B$};
    
    \pxcoordinate{0.49*\x}{0.33*\y}{A};
    \draw[->, >=latex, very thin] ([xshift=2.cm, yshift= 2.0cm] A) to (A);
    \draw ([xshift=3.5cm, yshift= 2.0cm] A) node {cusp};
\end{tikzgraphics}
\hspace{0.25em}
\def\x{484}
\def\y{471}
\begin{tikzgraphics}{\sizeFig\textwidth}{\x}{\y}{fig_value_function2}
    \pxcoordinate{0.55*\x}{1.05*\y}{A}; \draw (A) node {$y$};
    \pxcoordinate{-0.06*\x}{0.5*\y}{A}; \draw (A) node {$T$};
\end{tikzgraphics}
\vspace{-0.5cm}
\caption{(Left) The initial point is $q_0 = (-2,0)$. The abnormal geodesic with the cusp singularity is in green while the others geodesics in red are hyperbolic. We can see that the cusp singularity is the limit of self-intersecting hyperbolic geodesics. Besides, to reach the point $B$ from $q_0$, one has to use a hyperbolic self-intersecting geodesic. When this hyperbolic geodesic intersects the abnormal, the time is longer along the hyperbolic than the abnormal. At this intersection, the hyperbolic geodesic ceases to be optimal. (Right) The time minimal value function along the dashed segment from the left sub-graph. The discontinuity occurs at the intersection between the hyperbolic and abnormal geodesics. It is represented by the green dot, which is the time along the abnormal geodesic (see \cite{BRW:2021} for its evaluation).}
\label{fig:cusp_continuity_value_function_scl}
\vspace{-0.2cm}
\end{figure}

\paragraph{\textbf{Time minimal synthesis}}
We use the heading angle and the Clairaut relation to stratify the Lagrangian manifold $\mathcal{L}$ given by the image of $\exp(t\vv{\Htrue})(q_0, \cdot)$, see \cite[Chapter~10]{BonnardChyba:2003}, where $q_0$ is in the strong current domain.
One can also compute the time minimal synthesis (see Fig.~\ref{fig:optimal_synthesis_histical_example_scl}). In the strong current case, $q_0$ is not strongly locally controllable, i.e. $A(q_0,t)$ is not a neighborhood of $q_0$ for small-time.
The time $T(q_0)$ along the loop is the limit time such that $A(q_0,t)$ is a neighborhood of $q_0$.
Denote by $\Sigma(q_0)$ the adherence of the cut locus for geodesics starting in $q_0$ and contained in the adapted neighborhood. One has:

\begin{proposition}
Let $q_0$ in the strong current domain, then:
\begin{enumerate}
\item For $t>T(q_0)$, $A(q_0,t)$ is a neighborhood of $q_0$.
\item $\Sigma(q_0)$ is the union of two abnormal arcs, the one with a cusp point (corresponding to a conjugate point along the abnormal curve) being taken up to this cusp.    
\end{enumerate}
\end{proposition}

\def\sizeFig{0.5}
\begin{figure}[ht!]
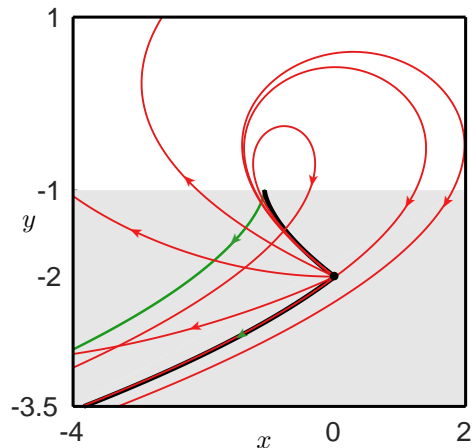

\centering
\def\x{493}
\def\y{470}
\begin{tikzgraphics}{\sizeFig\textwidth}{\x}{\y}{Flot_synthese}
    \pxcoordinate{0.55*\x}{1.0*\y}{A}; \draw (A) node {$x$};
    \pxcoordinate{0.04*\x}{0.5*\y}{A}; \draw (A) node {$y$};
    
    \pxcoordinate{0.485*\x}{0.539*\y}{A}; \draw (A) node[green, sloped, rotate=225] {\tiny \ding{228}};
    \pxcoordinate{0.5*\x}{0.755*\y}{A}; \draw (A) node[green, sloped, rotate=215] {\tiny \ding{228}};
    \pxcoordinate{0.655*\x}{0.4*\y}{A}; \draw (A) node[red, sloped, rotate=260] {\tiny \ding{228}};
    \pxcoordinate{0.863*\x}{0.45*\y}{A}; \draw (A) node[red, sloped, rotate=240] {\tiny \ding{228}};
    \pxcoordinate{0.945*\x}{0.45*\y}{A}; \draw (A) node[red, sloped, rotate=240] {\tiny \ding{228}};
    \pxcoordinate{0.4*\x}{0.737*\y}{A}; \draw (A) node[red, sloped, rotate=200] {\tiny \ding{228}};
    \pxcoordinate{0.385*\x}{0.4*\y}{A}; \draw (A) node[red, sloped, rotate=135] {\tiny \ding{228}};
    \pxcoordinate{0.27*\x}{0.518*\y}{A}; \draw (A) node[red, sloped, rotate=155] {\tiny \ding{228}};
\end{tikzgraphics}
\vspace{-0.5cm}
\caption{Minimal time optimal synthesis in an adapted neighborhood containing the limit loop. The initial point is $q_0 = (-2,0)$. The abnormal geodesics are in green. Hyperbolic geodesics are in red and the domain of strong current is in gray.}
\label{fig:optimal_synthesis_histical_example_scl}
\end{figure}

\section{Conclusion}

In this article, the results of \cite{BonnardKupka:1993} are applied to planar Zermelo navigation problems to give a neat analysis of the role of abnormal geodesics as limit curves of the accessibility set and to describe the fan shape of the balls with small radius. Generic singularities of the time minimal value function associated to normal directions are all well known due to earlier Withney classification \cite{Withney:1955}, see also \cite{Zermelo:1931} for the Hamiltonian frame. But based on the two case studies, we describe and provide a mathematical model of a (stable) singularity in the abnormal direction. This corresponds to a cusp singularity of the abnormal geodesics, taken as a limit points of self-intersecting normal geodesics. Moreover, in this situation, the time minimal value function in not continuous. Our study completes the contribution of \cite{BCW2019} devoted to the calculation of spheres with general radius, in the vortex case, when at the initial point the current is weak. It is a further step to analyze general navigation problems in the plane, in the case with a symmetry of revolution, combining geometric methods with numerical simulations. Also, it can serve as models to analyze singularities of the exponential mapping in the non-integrable case. In particular in the normal case using \eqref{eq:scl_forme_normal} and in the abnormal case using \eqref{eq:scl_forme_normal_abnormal} and \eqref{eq:scl_historical_example}. Besides, this gives the corresponding models of the singularities of the value function. See \cite{BRW:2021} for a construction of such a model in the general case.


\end{document}